\documentclass[12pt,leqno]{amsart}
\usepackage[dvips]{graphics}
\usepackage{amssymb}

\numberwithin{equation}{section}

\newtheorem{thm}{THEOREM}[section]

\newtheorem{lem}[thm]{Lemma}
\newtheorem{cor}[thm]{Corollary}

\newtheorem{quest}[thm]{PROBLEM}

 \theoremstyle{definition}
\newtheorem{defn}{Definition}[section]
\theoremstyle{remark}
\newtheorem{rem}{Remark}[section]

\newcommand{\tref}[1]{Theorem~\ref{#1}}
\newcommand{\cref}[1]{Corollary~\ref{#1}}

\newcommand{\lref}[1]{Lemma~\ref{#1}}

\newcommand{\diam}{\mathrm{diam}}

\newcommand{\R}{\mathbb{R}}
\newcommand{\C}{\mathbb{C}}
\newcommand{\J}{\mathbf{J}}

\newcommand{\Area}{\operatorname{Area}}

\newcommand{\md}{\operatorname{md}}

\newcommand{\ap}{\operatorname{ap}}
\newcommand{\apmd}{\ap\md}

\newcommand{\I}{\mathcal{I}}

\begin{document}
\pagebreak


\title{Energy  and area minimizers in metric spaces}

\author{Alexander Lytchak}

\address
  {Mathematisches Institut\\ Universit\"at K\"oln\\ Weyertal 86 -- 90\\ 50931 K\"oln, Germany}
\email{alytchak@math.uni-koeln.de}

\author{Stefan Wenger}

\address
  {Department of Mathematics\\ University of Fribourg\\ Chemin du Mus\'ee 23\\ 1700 Fribourg, Switzerland}
\email{stefan.wenger@unifr.ch}

\date{\today}

\thanks{The second author was partially supported by Swiss National Science Foundation Grant 153599}

\begin{abstract}
 We  show that  in the setting of proper  metric spaces one obtains a
solution of the classical two-dimensional  Plateau problem by minimizing
the energy, as in the classical case, once a definition of area (in the sense of convex geometry) has been chosen appropriately.
We prove the quasi-convexity of this new  definition of area. Under the assumption of a quadratic isoperimetric inequality we establish  regularity results for  energy minimizers
and improve Hoelder exponents of some area-minimizing discs.
 \end{abstract}

\maketitle

\maketitle
\renewcommand{\theequation}{\arabic{section}.\arabic{equation}}
\pagenumbering{arabic}
\section{Introduction}
\subsection{Motivation}
The classical Plateau problem concerns the existence and properties of  a disc of smallest area bounded by a given Jordan curve.
In a Riemannian manifold $X$, a solution of the Plateau problem  is obtained by a disc of minimal energy, where one minimizes over the set $\Lambda (\Gamma ,X)$ of  all maps  $u$ in the Sobolev space
$W^{1,2} (D,X)$, whose boundary $tr(u):S^1\to X$ is a reparametrization of the given Jordan curve  $\Gamma$.  This approach has the additional useful feature that the area minimizer obtained in this way is automatically conformally parametrized.

Recently, the authors of the present article generalized the classical
Plateau problem to the setting of arbitrary proper metric spaces in
 \cite{LW}. In particular, they proved existence of area minimizing discs
with prescribed boundary in any proper metric space and with respect to
any quasi-convex definition of area (in the sense of convex geometry).
It should be noted that the classical approach (described above) to the
Plateau problem cannot work literally in the generality of metric
spaces. This is due to the fact that there are many natural but
different definitions of area and of energy. Moreover, different
definitions of area may give rise to different minimizers as was shown
in  \cite{LW}. Finally, the presence of normed spaces destroys any hope of
obtaining a conformal area minimizer and the inevitable lack of
conformality  is the source of difficulties when trying to compare or
identify minimizers of different energies and areas.

One of the principal aims of the present article is to show   that
the classical approach  (of minimization of the area via the simpler
minimization of the energy) does in fact work for some definitions of
energy and area. As a byproduct  we obtain new definitions of area which
are quasi-convex  (topologically semi-elliptic
in the language of \cite{Iva09}), which might be of some  independent
interest in convex geometry.

\subsection{Energy and area minimizers}
For a metric space $X$, the Sobolev space   $W^{1,2} (D,X)$ consists of all
measurable, essentially separably valued maps $u:D\to X$   which admit  some function $g\in L^2 (D)$ with the following property (cf. \cite{Res97}, see also   \cite{HKST15}): For  any $1$-Lipschitz function $f:X\to \R$ the composition $f\circ u$ lies in the classical Sobolev space $W^{1,2} (D )$, and the norm of the gradient of $f\circ u$ is bounded from above by $g$ at almost every point of $D$.  In $L^2(D)$ there exists   a unique minimal function $g$ as above, called the  \emph{generalized gradient}  of $u$. This generalized gradient $g_u$ coincides with the \emph{minimal weak upper gradient}  of a representative of $u$ in the sense of \cite{HKST15}.    The square of the $L^2$-norm of this \emph{generalized gradient} $g_u$ is the  Reshetnyak energy of $u$, which we denote by $E_+^2 (u)$.
 A different but equivalent definition of the Sobolev space $W^{1,2} (D,X)$ is due to Korevaar-Schoen (\cite{KS93}) and comes along with another definition of energy $E^2(u)$  generalizing the classical  Dirichlet energy.

 If $X$ is a Riemannian manifold then $g_u (z)$  is just the point-wise sup-norm of the weak differential $Du (z)$  for almost all $z\in D$. The Dirichlet-Korevaar-Schoen energy  $E^2 (u)$ is obtained in this case by integrating over $D$ the sum of squares of eigenvalues of $Du(z)$.
 It is the heart of the classical approach to Plateau's problem by Douglas and Rado, extended by Morrey to Riemannian manifolds,  that any minimizer of the Dirichlet energy $E^2$ in $\Lambda (\Gamma , X)$ is conformal and minimizes the area in $\Lambda (\Gamma, X)$.

 Turning to general proper metric spaces $X$, we recall from \cite{LW} that for any Jordan curve $\Gamma$ in  $X$ one can
 find  minimizers of $E^2$ and $E^2_+$  in the set $\Lambda (\Gamma, X)$,  whenever  $\Lambda (\Gamma, X)$  is not empty.
 The first special case of our main result \tref{thmmain}  identifies any  minimizer of the Reshetnyak energy $E_+^2$  in $\Lambda (\Gamma, X)$ as a minimizer of the  inscribed Riemannian area $\mu^i$ investigated by Ivanov in \cite{Iva09}, see also Subsections~\ref{subsecarea} -- \ref{vergleich} below.

\begin{thm}  \label{thm2}
Let $\Gamma$ be   any  Jordan curve  in a proper metric space $X$. Then every map $u\in \Lambda (\Gamma, X)$ which minimizes the Reshetnyak
energy $E_+ ^2$ in  $\Lambda (\Gamma ,X)$ also minimizes the $\mu^i$-area  in $\Lambda (\Gamma ,X)$.
\end{thm}

Any minimizer  $u$  of the  Reshetnyak energy as in \tref{thm2}  is  $\sqrt 2$-quasiconformal.  This means, roughly speaking, that $u$ maps infinitesimal balls to ellipses of aspect ratio at most $\sqrt{2}$, see \cite{LW} and Subsection \ref{subsecconst} below. We emphasize that our notion of quasiconformal map is different from the notion of quasiconformal homeomorphism studied in the field of quasiconformal mappings. For any map $v\in W^{1,2} (D,X)$ there is an  energy-area inequality $E^2_+ (v)\geq \Area _{\mu ^i} (v)$; and for any  $u$ as in \tref{thm2} equality holds.

We find a similar phenomenon in the case of the more classical Korevaar-Schoen energy $E^2$, which generalizes the Dirichlet energy from the Riemannian to the metric setting. However, the corresponding \emph{Dirichlet definition of area} $\mu^D$  seems to be new, see Subsection~\ref{subsec+}.


\begin{thm}  \label{thm1}
 There exists a quasi-convex definition of area $\mu ^{D}$
such that the following holds true. For   any  Jordan curve $\Gamma$   in a proper metric space $X$, and   for any
map $u\in \Lambda (\Gamma, X)$ with minimal Korevaar-Schoen energy $E ^2(u)$ in  $  \Lambda (\Gamma ,X)$,  this map $u$ minimizes the $\mu ^D$-area  in $\Lambda (\Gamma ,X)$.
\end{thm}

Recall that quasi-convexity of the definition of area is a very important feature in the present context, since it is equivalent to the  lower semi-continuity of the corresponding area functional  in all Sobolev spaces \cite{LW}, Theorem 5.4, and therefore, closely related to the question  of the existence of area minimizers.

 In order to describe the definition of area $\mu ^{D}$, we just need to fix  the values of the $\mu ^D$-areas of one subset in every normed plane $V$. Considering the subset to be the ellipse arising as the image $L(D)$ of a linear map $L: \R^2\to V$
(see Section \ref{sec2} below), this value
$\Area _{\mu ^D} (L) $ equals
\begin{equation} \label{equation1}
\Area_{\mu ^D} (L)= \frac 1 2  \inf \{ E^2 (L\circ g)\; | \; g\in {\rm SL}_2  \}.
\end{equation}
%
For any Sobolev map $v\in  W^{1,2} (D,X)$  the  energy-area  inequality   $E^2(v) \geq 2\cdot \Area_{\mu^D} (v)$ holds true, with equality for any
  any minimizer $u$ as in \tref{thm1}.   The minimizers  in \tref{thm1} are $Q$-quasiconformal
with  the non-optimal  constant $Q= 2\sqrt 2 + \sqrt 6$ (\cite{LW} and Subsection \ref{subsec+} below).
 An answer to  the following question would shed light on the  structure of energy minimizers from \tref{thm1}, cf. \cite{Milman}, p.723 for the "dual" question.

\begin{quest}
For which $g\in {\rm SL}_2$ is the infimum in \eqref{equation1}  attained?
Is it possible to describe the measure $\mu ^{D}$ appearing in \tref{thm1} in a geometric way? What is the optimal quasiconformality constant of the minimizers of the Korevaar-Schoen energy?
\end{quest}

All definitions of  area of Sobolev maps agree with the parametrized Hausdorff area if $X$ is a Riemannian manifold, a space with one-sided curvature bound
 or, more generally,
any space with the property (ET) from \cite{LW}, Section 11. In this case, \tref{thm1} directly generalizes  the classical result of Douglas-Rado-Morrey.

Our results apply to all other \emph{quasi-convex definitions of energy}, see \tref{quasiarea}. We refer
to  Section \ref{sec2} for the exact definitions and mention as a particular example
linear combinations $a\cdot E^2+ b\cdot E_+ ^2+ c\cdot \Area _{\mu}$, where
$a,b,c\geq 0$ with $a^2+b^2>0$ and where $\mu$ is some quasi-convex definition of area.  For any such energy $E$ there exists a quasi-convex definition of area $\mu ^E$ such that a minimizer of $E$ automatically provides a quasiconformal minimizer of $\mu ^E$ as in \tref{thm2} and
\tref{thm1}.
The definition of area  $\mu ^E$ is  given similarly to \eqref{equation1}.

\begin{rem}
We would like to mention a related method of obtaining an area-minimizer
for any quasi-convex definition of area $\mu$.  In the Riemannian case this idea can be found in \cite{HM}, cf. \cite{Dierkes-et-al10}, Section 4.10: Consider the energy $E_{\epsilon} =\epsilon E_+^2 + (1-\epsilon ) \Area _{\mu}$.  Then a minimizer $u_{\epsilon}$ of $E_{\epsilon}$ in $\Lambda (\Gamma ,X)$  can be found in the same way as the minimizer of $E_+^2$.      This minimizer is automatically
$\sqrt 2$-quasiconformal and minimizes the area functional  $(1-\epsilon ) \Area _{\mu}+
\epsilon \Area _{\mu ^i}$ in $\Lambda (\Gamma ,X)$.   Due to the quasiconformality these minimizers have uniformly bounded energy. Therefore one can go to the limit (fixing three points in the boundary circle) and obtain a minimizer of $\Area _{\mu}$.
\end{rem}



This remark also shows that the set of quasi-convex areas obtained via the minimization of energies as in \eqref{equation1} is a dense convex subset
in the set of all quasi-convex definitions of area. It seems to be a natural question  which
definitions of area correspond  in this way to some energies. In particular, if it is the case for the most famous Hausdorff, Holmes-Thompson and Benson definitions of area.

\begin{rem} \label{remark}
From \tref{quasiarea} and \tref{thm2} one can deduce the quasi-convexity of the inscribed Riemannian area $\mu ^i$. However, a much stronger convexity property of this area has been shown  in \cite{Iva09}.
\end{rem}

\subsection{Regularity of energy minimizers}
In the presence of quadratic isoperimetric inequalities the regularity results for area minimizers obtained in \cite{LW}  imply regularity of energy minimizers,
once we have identified energy minimizers as area minimizers in \tref{thmmain}.
Recall that a complete metric space $X$ is said to admits a $(C,l_0)$-quadratic isoperimetric inequality with respect to a definition of area $\mu$ if for every Lipschitz curve $c:S^1\to X$ of length
$l\leq l_0$ there exists some $u\in W^{1,2} (D,X)$ with $$\Area _{\mu} (u) \leq C\cdot l^2$$ and such that the trace $tr (u)$ coincides with $c$.
We refer to \cite{LW} for a discussion of this property satisfied by many interesting classes of metric spaces.
If $\mu$ is replaced by another definition of area $\mu '$ then in the definition above only the constant $C$  will change and it will be changed by at most by the factor $2$.  If the assumption is satisfied for some triple $(C,l_0,\mu)$ we say that $X$ satisfies a \emph{uniformly local quadratic isoperimetric inequality}.

As far as qualitative statements are concerned  the constants and the choice of the area do not play any role.  As a consequence of \tref{thm1} and  the regularity results for area minimizers in \cite{LW} we easily deduce continuity up to the boundary and local Hoelder continuity in the interior for
 all  energy minimizers in $\Lambda (\Gamma ,X)$ for any quasi-convex definition of energy.
We refer to \tref{thmgenreg} for the precise statement.

\subsection{Improved regularity of $\mu$-minimal discs}   We can use \tref{thm2} to slightly improve the regularity results for solutions of the Plateau problem obtained in \cite{LW}.
Assume again that $\Gamma$ is a Jordan curve in a proper metric space $X$ and let $\mu$ be a definition of area.
 We introduce the following
   \begin{defn}
   We say that a map $u\in \Lambda (\Gamma ,X)$ is $\mu$-minimal if it minimizes the $\mu$-area in $\Lambda (\Gamma ,X)$, and if
   it has minimal Reshetnyak energy $E_+^2$ among all such minimizers of the $\mu$-area.
\end{defn}

Due to \tref{thm2}, for the inscribed Riemannian definition of area $\mu=\mu ^i$, a $\mu$-minimal disc is just a minimizer
of the Reshetnyak energy $E_+^2$ in $\Lambda (\Gamma, X)$.
It follows from \cite{LW} that, for any  quasi-convex $\mu$, one finds  some  $\mu$-minimal disc  in any non-empty $\Lambda (\Gamma , X)$.
Moreover, any such $\mu$-minimal map is $\sqrt 2$-quasiconformal.
 Assume further that $X$ satisfies the $(C,l_0,\mu )$-quadratic  isoperimetric inequality. In \cite{LW}, we used the quasiconformality to deduce that any such map
 has a locally $\alpha$-Hoelder continuous representative with $\alpha=\frac 1 {8 \pi C}$.  However, $\mu$-minimal maps satisfy a stronger
  infinitesimal condition than $\sqrt 2$-quasiconformality, and this can be used to improve $\alpha$ by a factor of $2\cdot q(\mu) \in [1,2]$ depending on the definition of area $\mu$. The number  $q(\mu )$ equals $1$ for the maximal definition of area $\mu=\mu^i$.  For other definitions of area $\mu$, the number $q(\mu)$ is smaller than $1$ and measures the maximal possible deviation of $\mu$ from $\mu ^i$, see \eqref{eq-def-qmu}.
   For instance,
  $q(\mu ^b ) =\frac{\pi}{4}$  for the Hausdorff area $\mu ^b$.  Thus the following  result improves the above Hoelder exponent by $2$ in the case of the inscribed Riemannian definition of area $\mu  =\mu ^i$ and by $\frac \pi 2$ in the case of the Hausdorff area $\mu =\mu ^b$:

  \begin{thm} \label{optregul}
  Let $\Gamma$ be a Jordan curve in a proper metric space $X$. Assume that $X$ satisfies the $(C,l_0,\mu )$-quadratic  isoperimetric inequality
  and let $u$ be a $\mu$-minimal disc in $\Lambda (\Gamma ,X)$.  Then  $u$ has a
  locally $\alpha$-Hoelder continuous representative  with  $\alpha = q(\mu)\cdot\frac{1}{4 \pi C}$.
  \end{thm}
    For  $\mu =\mu ^i$ we get  the  optimal Hoelder exponent $\alpha =\frac 1 {4\pi C}$  as examples of cones over small circles show
  (see \cite{MR02} and \cite{LW}, Example 8.3).


\subsection{Some additional comments}
The  basic ingredient in the proof of \tref{thm2} and its generalization \tref{thmmain} is the localized version of the classical conformality of energy minimizers. This was already used in \cite{LW}. This idea  shows that almost all  (approximate metric) derivatives  of  any minimizer $u$ of the energy $E$ in $\Lambda (\Gamma ,X)$ have to minimize the energy in their corresponding ${\rm SL}_2$-orbits, as in \eqref{equation1}.


The proof of the quasi-convexity of $\mu ^D$, generalized by  \tref{quasiarea}, is achieved by applying an idea from \cite{Jost}. We obtain  a special parametrization of arbitrary  Finsler discs by minimizing the
energy under additional topological constraints.  This idea  might  be of independent interest as it provides canonical parametrizations of any sufficiently regular surface.

\bigskip

{\bf Acknowledgements:} We would like to thank Frank Morgan for useful comments.

\section{Preliminaries} \label{sec2}
 \subsection{Notation}   By $(\R^2, s)$ we denote the plane equipped with a semi-norm $s$.  If $s$ is not specified, $\R^2$ is always considered with its canonical Euclidean norm, always  denoted by $s_0$. By $D$
we denote the open unit disc in the Euclidean plane $\R^2$ and by $S^1$ its boundary, the unit circle. Integration on open subsets of $\R^2$ is always performed with respect to the Lebesgue measure, unless otherwise stated. By $d$ we denote distances in metric spaces. Metric spaces appearing in this note will be assumed complete.
A metric space is called proper if its closed bounded subsets are compact.

\subsection{Seminorms and convex bodies}\label{subsecseminorm}
By $\mathfrak S_2$ we denote the proper metric space of seminorms on $\R^2$ with the distance given by $d_{\mathfrak S_2} (s,s')= \max _{v\in S^1} \{|s(v)-s'(v)| \}$.  A seminorm $s \in \mathfrak S_2$ is  \emph{$Q$-quasiconformal} if for all $v,w\in S^1$ the inequality $s(v) \leq Q\cdot s(w)$ holds true.
A convex body $C$ in $\R^2$ is a compact convex subset with non-empty interior. Convex, centrally symmetric bodies are in one-to-one correspondence with unit balls of norms on $\R^2$.
Any convex body $C$  contains a unique ellipse of largest area, called the \emph{Loewner ellipse} of $C$. The convex body is called \emph{isotropic}
if its Loewner ellipse is a Euclidean ball (cf. \cite{Milman}).  We call a seminorm $s\in  \mathfrak S_2$ isotropic if it is the $0$ seminorm, or
if $s$ is a norm and its  unit ball $B$ is isotropic. In the last  case the Loewner ellipse of $B$ is a multiple $t\cdot \bar D$ of the closed unit disc.
By John's theorem (cf. \cite{AlvT04}) $B$ is contained  in $\sqrt 2 \cdot t \cdot \bar D$. Therefore, every isotropic seminorm is $\sqrt 2$-quasiconformal.

\subsection{Definitions of area} \label{subsecarea}
While there is an essentially unique natural way to measure areas of Riemannian surfaces, there are many  different ways to measure areas of Finsler surfaces,
some of them more appropriate for different questions.  We refer the reader to
\cite{Iva09}, \cite{Ber14}, \cite{AlvT04} and the literature therein for  more information.

A definition of area $\mu$ assigns a multiple $\mu _V$ of the Lebesgue measure on any $2$-dimensional normed space $V$, such that natural assumptions are fulfilled.  In particular, it assigns the number
$\J^{\mu} (s)$, \emph{the $\mu$-Jacobian} or \emph{$\mu$-area-distortion}, to any seminorm $s$ on $\R^2$ in the following way. By definition,  $\J ^{\mu} (s)=0$ if the seminorm is not a norm. If $s$ is a norm then  $\J ^{\mu} (s)$ equals the  $\mu _{(\R^2,s)}$-area  $\mu _{(\R^2,s)}  (A)$  of the unit   Euclidean square  $A\subset \R^2$. Indeed, the choice of the definition of area is equivalent to a choice of the \emph{Jacobian} in the following sense.
\begin{defn}
A (2-dimensional definition of) Jacobian  is a map $\J: \mathfrak S_2 \to [0,\infty )$  with the following properties:
\begin{enumerate}
\item Monotonicity:    $\J(s)\geq \J(s')$ whenever $s\geq s'$;
\item Homogeneity: $\J(\lambda \cdot s)= \lambda ^2 \cdot \J(s)$ for all $\lambda \in [0,\infty)$;
\item  ${\rm SL}_2$-invariance $\J(s\circ T)=\J(s)$ for any $T\in {\rm SL}_2$;
\item Normalization:  $\J (s_0)= 1$.
\end{enumerate}
\end{defn}

The  properties (2) and (3) can be joined to the usual transformation rule for the area: $\J(s\circ T)= |\det (T)| \cdot \J(s)$.  It  follows that
$\J(s)= 0$ if and only if the seminorm $s$ is not a norm. Moreover, properties
(1)-(3) imply that $\J$ is continuous. This is due to the following crucial  fact: If norms $s_i$ converge to a norm $s$ in $\mathfrak S_2$  then, for any $\epsilon >0$ and all large $i$,  the inequalities $(1-\epsilon) \cdot s_i \leq s \leq (1+\epsilon ) \cdot s_i$ hold true.

A definition of area $\mu$ gives rise to a Jacobian $\J^{\mu}$ described above. On the other hand, any Jacobian $\J:\mathfrak S_2\to [0,\infty )$ provides a unique definition of area $\mu ^{\J}$ in the following way.
On any $(\R^2 ,s)$ the definition of area $\mu ^{\J}$ assigns the $\J(s)$-multiple of the Lebesgue area of $\R^2$. For any normed plane $V$, we choose  a linear isometry to some $(\R^2, s)$ and pull back the corresponding measure from $(\R^2, s)$ to $V$.  By construction, the assignments $\mu \to \J ^{\mu}$
and $\J\to \mu ^{\J}$ are inverses of each other.

\begin{rem}
We refer to another similar geometric  interpretation of a definition of area  discussed in \cite{Ber14}.
\end{rem}

There are many non-equivalent definitions of area/Jacobian. Any two of them differ at most by a factor of $2$, due to John's theorem, \cite{AlvT04}.  The most prominent examples are the Busemann (or Hausdorff) definition  $\mu ^b$, the Holmes-Thompson definition $\mu ^{ht}$,  the Benson (or Gromov $mass ^*$) definition $m^{\ast}$  and the inscribed Riemannian (or Ivanov) definition $\mu ^i$.  We refer to \cite{AlvT04} for  a thorough discussion of these examples and of the whole subject; and to  \cite{Iva09}, \cite{BI13}, \cite{Ber14} for recent developments. Here, we just mention the Jacobians of these four examples (cf. \cite{Ber14}). In the subsequent examples, $B$ will always denote the unit ball of the  normed plane
$(\R^2,s)$.
\begin{enumerate}
\item The Jacobian  $\J^b$ corresponding to the Hausdorff (Busemann) area $\mu^b$
equals $\J ^b(s)=\frac {\pi} {|B|}$, where $|B|$ is  the Lebesgue area of  $B$.
\item The  Jacobian $\J^{ht}$ corresponding to the  Holmes-Thomspon area $\mu ^{ht}$ equals
$\J^{ht}(s)= \frac {|B^*|} {\pi}$, where $|B^*|$ is the Lebesgue area of the
unit ball $B^*$ of the dual norm $s^*$ of $s$.
  \item The Jacobian $\J^*$ corresponding to Benson (Gromov mass$^{\ast}$) definition  of area $m^*$ equals
$\J ^*(s)= \frac 4 {|P|}$, where $|P|$  is the Lebesgue area of a
parallelogram $P$ of smallest area which contains $B$.
\item  The Jacobian $\J^i$ corresponding to  the inscribed Riemannian definition of area $\mu ^i$ equals $\J^i(s) =\frac {\pi} {|L|} $, where $|L|$ is Lebesgue area of  the Loewner ellipse  of  $B$.
\end{enumerate}

\subsection{Comparision of the definitions of area} \label{vergleich}
Below we denote by $|C|$ the Lebesgue area of a subset $C\subset \R^2$.
Let $s$ be a norm on $\R^2$, let $B$ be its unit ball and let $L\subset B$ denote the Loewner ellipse of $B$.  If $s_L$ denotes the norm whose unit ball is
$L$, then $s\leq s_L$ and $s_L$ is Euclidean. Thus, for any definition of area
$\mu $ with Jacobian $\J ^{\mu}$ we have $\J^{\mu} (s_L) = \frac {\pi} {|L|}$ and
$\J ^{\mu}(s) \leq \J^{\mu}(s_L)$.

For the inscribed Riemannian area $\mu ^i$ and its Jacobian $\J^i$ we have equality $\J^{i} (s) =\J^{i} (s_L)$ in the above inequality. Hence, for any other definition of area $\mu$ we must have $ \J^{i} \geq \J^{\mu}$. In particular, the inscribed Riemannian area is the largest definition of area.
On the other hand, by John's theorem,  $\J^i \leq 2\J ^{\mu}$.

We set 
\begin{equation}\label{eq-def-qmu}
q(\mu) := \inf \frac{\J^{\mu } (s)}{\J^i (s)},
\end{equation}
 where $s$ runs over all norms on $\R^2$. As we have  just observed, $q(\mu^i )=1$
and $1/2\leq q(\mu) < 1$ for any other definition of area $\mu$.

\begin{lem}
For the Hausdorff area $\mu ^b$ we have $q(\mu ^b) = \frac{\pi}{4}$.
\end{lem}

\begin{proof}
Let $B$ be the unit ball of the norm $s$ on $\R^2$.  In order to compare $\J^i (s)$ and $\J^b (s)$ we just need
to evaluate $\mu^i$ and $\mu ^b $ on $B$. For the Busemann definition of area we have  $\mu ^b (B)=\pi$.
On the other hand,  $\mu^i (B)= \pi \cdot \frac {|B| } {|L|}$, where $L$ is the Loewner ellipse of
$B$. The \emph{volume ratio} $\frac {|B|} {|L|}$ is maximal when $B$ is a square, see \cite{Bal97}, Theorem 6.2, in which case it is equal to  $\frac 4 {\pi}$.
\end{proof}

Since we will not need further statements about the function $q$ we just summarize here some properties without proofs.
For any definition of area $\mu$, there exists a norm  $s$ with $q(s) \cdot \J^{\mu} (s)= \J^i (s)$. Moreover, using John's theorem
one can show that this norm $s$ can be chosen to have a square or a hexagon as its unit ball.   One can show that
$q(\mu ^{ht} ) = \frac{2}{\pi}$, where again on the supremum norm  $s_{\infty}$  the difference between  $\J^i$ and $\J^{ht}$ is maximal. Finally,  for
Gromov's  definition of area $m^{\ast}$  one can show that $q(m^{\ast} )=\frac{\sqrt  3}{2}$. Here the maximal deviation of $\J^i$ from
$\J^{\ast}$ is achieved for the norm whose unit ball is a regular hexagon.

\subsection{Definitions of energy} \label{subsecdef}
An assignment of a definition of area or Jacobian  is essentially equivalent to the assignment of an area functional on all Lipschitz and Sobolev maps defined on domains in $\R^2$, see below.  Similarly, the choice of an energy functional is essentially equivalent to the following choice of a \emph{definition of energy}:

\begin{defn} A ($2$-dimensional conformally invariant) \emph{definition of energy}  is  a continuous  map   $\I:\mathfrak  S_2\to [0,\infty )$ which has the following properties:
\begin{enumerate}
\item Monotonicity:    $\mathcal I(s)\geq \mathcal I(s')$ whenever $s\geq s'$;
\item Homogeneity: $\mathcal I(\lambda \cdot s)= \lambda ^2 \cdot \I(s)$ for all $\lambda \in [0,\infty)$;
\item  ${\rm SO}_2$-invariance: $\I(s\circ T)=\I(s)$ for any $T\in {\rm SO}_2$;
\item Properness: The set $\I^{-1} ([0,1])$ is compact in $\mathfrak S_2$.
\end{enumerate}
\end{defn}

Due to properness and homogeneity, we have $\I(s)=0$ only for $s=0$.
The properness of $\I$ implies that a definition of energy is \emph{never} ${\rm SL}_2$-invariant, in contrast to a definition of area.
The set of all definitions of energy is a convex cone. Moreover, for any Jacobian $\J$, any definition
of area $\I$ and any $\epsilon >0$ the map $I_{\epsilon} := \J+ \epsilon \cdot \I$ is a definition of energy. Thus the closure (in the topology of locally uniform convergence) of the set of definitions of energy contains all definitions of area.


The following two  definitions of energy are most prominent:  the Korevaar-Schoen-Dirichlet  energy $I^2$ given by $$I^2(s)=\frac 1 {\pi}  \int _{S^1} s(v)^2 dv$$
and the Reshetnyak energy $$I_+^2 (s)= \sup \{ s(v)  ^2 |v\in S^1 \}. $$
Due to properness and homogeneity any two definitions of energy are comparable: For any  definition of energy  $\I$  there is a constant $k_{\I} \geq 1$, such that
\begin{equation} \label{2ener}
\frac 1 {k_{\I}} \cdot {\I} \leq I_+ ^2 \leq k_{\I} \cdot \I.
\end{equation}

\subsection{Energy and area of Sobolev maps} \label{subsecsob}
We assume some experience with Sobolev maps and refer  to \cite{LW}
and the literature therein.  In this note we consider only Sobolev maps   defined on bounded open domains $\Omega \subset \R^2$. Let $\Omega$ be such a domain and let $u\in W^{1,2} (\Omega,X)$ be a Sobolev map with values in $X$.
Then $u$ has an approximate metric derivative
at almost every point $z\in \Omega$ (\cite{Kar07},\cite{LW}), which is  a
seminorm on $\R^2$ denoted by $\apmd u_z$. When $\apmd u_z$  exists, it is the unique seminorm $s$ for which the following approximate limit is $0$:
$$\ap \lim _{y\to z}  \frac {d(u(z),u(y))- s(y-z)} {|y-z|} =0.$$
We refer the reader to \cite{LW}, \cite{Kar07} and mention here only that in the case of locally Lipschitz maps $u$,  the approximate  metric derivative is just the metric derivative defined by Kirchheim
(\cite{Kir94}, cf. also \cite{AK00}, \cite{Iva09}).
If the target space $X$ is a Finsler manifold then the approximate  metric derivative at almost all points $z$ is equal to $|D_z u|$, where $D_z u$ is the usual (weak) derivative and $|\cdot |$ is the given norm on  the tangent space $T_{u(z)} X$.
A map $u\in W^{1,2} (\Omega,X) $ is called \emph{$Q$-quasiconformal}  if  the seminorms $\apmd u_z \in \mathfrak S_2$ are $Q$-quasiconformal for almost all $z\in \Omega$.

For a definition of energy $\I$, the $\I$-energy of a map $u\in W^{1,2} (\Omega ,X)$ is given by  $$E_{\I}(u):= \int _{\Omega}   \I( \apmd u_z) dz$$
This value is well-defined and finite
for any $u \in W^{1,2} (\Omega ,X)$, due to \eqref{2ener}.   If $\I$  is the Korevaar-Schoen  definition of energy  $I^2$, respectively the Reshetnyak definition of energy $I_+^2$  then $E_{\I}(u)$ is the Korevaar-Schoen  respectively the Reshetnyak energy of $u$ described in \cite{KS93}, \cite{Res97} and in the introduction. We will denote $E_{\I}$ in these cases  as before by $E^2$  and  $E_+^2 $, respectively.

Similarly, given a definition of area $\mu$ and the corresponding Jacobian
$\J^  {\mu} $ one obtains the $\mu$-area of $u$ by integrating $\J ^{\mu} (\apmd u_z)$ over $\Omega$. We will denote it by
$$\Area _{\mu} (u)  :=\int _{\Omega}   \J^{\mu} ( \apmd u_z) dz$$
Pointwise comparision of $\mu$ with the inscribed Riemannian definition of area $\mu^i$ discussed in Subsection \ref{vergleich} gives us
for any Sobolev map $u$:
\begin{equation} \label{areafunct}
q(\mu)^{-1} \cdot \Area _{\mu } (u) \geq \Area _{\mu^i} (u) \geq \Area _{\mu} (u).
\end{equation}

\subsection{Quasi-convexity} \label{subsecquas}
A definition of energy  $\I: \mathfrak{S}_2 \to[0,\infty)$ is called \emph{quasi-convex}  if  linear $2$-dimensional subspaces of normed vector spaces have minimal $\I$-energy. More precisely, if  for every finite dimensional normed space $Y$ and every linear map $L: \R^2\to Y$ we have
 \begin{equation}\label{eq1}
E_ {\I} (L|_D) \leq  E_{\I} (\psi)
 \end{equation}
 for every smooth immersion $\psi:  \bar  D\to Y$ with $\psi|_{\partial D} = L|_{\partial D}$.

Similarly, one defines the quasi-convexity of a definition of area with corresponding functional $\J:\mathfrak S _2 \to [0,\infty )$, see Section 5 in \cite{LW}.
 As has been shown in \cite{LW},
in extension of the classical results (cf. \cite{AF84}), a definition of energy
is quasi-convex if and only if the map $u\mapsto E_{\I}(u)$ is  semi-continuous
on any Sobolev space $W^{1,2} (\Omega, X)$ (with respect to  $L^2$-convergence). Similarly,
 the quasi-convexity of a definition of area $\mu$ is equivalent to the semi-continuity property of the $\mu$-area on all Sobolev spaces $W^{1,2} (\Omega, X)$.

Recall that the Reshetnyak and Korevaar-Schoen definitions of energy are quasi-convex (\cite{KS93}, \cite{Res97}, \cite{LW}). The four definitions of area mentioned in Subsection \ref{subsecarea} are quasi-convex as well (\cite{Iva09}, \cite{BI13}, \cite{AlvT04}, \cite{LW}).

We dwell a bit discussing the properties of  a definition of area $\mu$ which is not  quasi-convex  (cf.~\cite{Mor52}).  Let $L:\R^2\to Y$ be a linear map to a finite-dimensional
normed vector space and let $\psi :\bar D\to Y$ be a smooth map which coincides with $L$ on $S^1$ and satisfies
 $$\Area _{\mu} (\psi) < \Area _{\mu} (L|_D).$$
By enlarging $Y$ if needed and by using  a general position argument we can assume that $\psi$ is a diffeomorphism onto its image.
  Now we can obtain a special sequence
of maps $\psi _m:\bar D\to Y$  converging to $L:\bar D\to Y$ and violating the semi-continuity property in the following way. The map $\psi _m$ differs from $L$ on $\delta \cdot m ^2$  disjoint balls of radius $m^{-1}$, where $\delta>0$ is a sufficiently small, fixed constant.  The difference between $\psi _m$
and $L$ on any of these balls is given by the corresponding translate of $\psi$, rescaled by the factor $m^{-1}$.

Then there is a number $K>0$,  such that any of the maps $\psi _m$ is
biLipschitz with the same biLipschitz constant $K$.  The maps $\psi _m$ converge
uniformly to the linear map $L$. Finally, for  $\epsilon = \Area _{\mu } (L)  - \Area _{\mu} (\psi )$, we deduce   $\Area _{\mu } (\psi _m)  =\Area _{\mu } (L) - \delta \cdot \epsilon $   for all $m$. In particular, $\Area_ {\mu} (L)>\lim _{m\to \infty} (\Area {\mu  } (\psi _m))$.

\section{Area definition corresponding to an energy}  \label{sec3}
\subsection{General construction}

Let now $\I$ be any definition of energy. Consider
 the function $\hat J: \mathfrak S_2\to [0,\infty )$:
 $$\hat J(s):= \inf \{ \I(s\circ T) | T\in {\rm SL}_2 \}$$ given by the infimum of
$\I$ on the ${\rm SL}_2$-orbit of $s$. Due to the properness of $\I$, the infimum in the above equation  is indeed a minimum, unless the seminorm is not a norm. On the other hand, if $s$ is not a norm then the ${\rm SL}_2$-orbit of $s$ contains the $0$ seminorm in its closure, and we get $\hat J(s)=0$.
By construction, the function $\hat J:\mathfrak S_2\to [0,\infty )$ is ${\rm SL}_2$-invariant. Since $\I$ is monotone and homogeneous, so is $\hat J$.
Finally,  $\hat J(s_0)$ is different from $0$. Thus, setting the constant $\lambda _{\I}$ to be
 $\frac 1 {\hat J (s_0)}$,  we see that $\J ^{\I} (s) := \lambda _{\I} \cdot  \hat J (s) $
is a definition of a Jacobian in the sense of the previous section.
The definition of area which corresponds to the Jacobian $\J^{\I}$ will be denoted by $\mu ^{\I}$.
By construction,
\begin{equation} \label{generalineq}
\J ^{\I} (s) \leq \lambda _{\I} \cdot \I(s)
\end{equation}
 with equality if and only if $\I$ assumes the minimum on the ${\rm SL}_2 $-orbit of $s$ at the seminorm $s$.

\begin{defn}
We will call a seminorm $s$ minimal for the definition of energy $\I$, or just $\I$-minimal, if $\I(s) \leq \I(s\circ T)$ for all $T\in {\rm SL}_2$.
\end{defn}

Thus  a seminorm $s$ is $\I$-minimal if and only if   we have equality in the inequality \eqref{generalineq}.
By homogeneity and continuity, the set of all $\I$-minimal seminorms is  a closed cone. Any $\I$-minimal seminorm is either a  norm or the trivial seminorm $s=0$. We therefore deduce by a limiting argument:

\begin{lem} \label{quasicon}
There is a number $Q_{\I} >0$ such that  any $\I$-minimal seminorm $s$   is $Q_{\I}$-quasiconformal.
\end{lem}



\subsection{The Reshetnyak energy and the inscribed Riemannian area} \label{subsecconst}
We are going to discuss the application of the above construction to the main examples. In order to describe the Jacobian $\J^{\I}$, the normalization and the quasiconformality constants $\lambda _{\I}, Q_{\I}$ induced by a definition of energy $\I$, it is crucial to understand $\I$-minimal norms.  By  general symmetry reasons one might expect that $\I$-minimal norms are particularly round. Our first result, essentially contained in \cite{LW}, confirms this expectation for the Reshetnyak energy:


\begin{lem} \label{resiso}
Let $\I=I_+^2$ be the Reshetnyak definition of energy.  A seminorm $s\in \mathfrak S_2$ is $I^2 _+$-minimal if and only if $s$ is isotropic in the sense of Subsection \ref{subsecseminorm}.
%
\end{lem}

\begin{proof}
For seminorms which are not norms the statement is clear. Thus we may assume that $s$ is a norm. After rescaling,
 we may assume $I^2_+(s)=1$. Hence $1=\sup \{ s(v), v\in S^1\}$, and  $\bar D$ is the largest Euclidean disc contained in the unit ball $ B$ of the norm $s$.

   Assume that $s$ is $I_+ ^2$-minimal and  $\bar D$ is not the Loewner ellipse of $B$. Then there exists   an area increasing linear map $A:\R^2\to \R^2$ such that $B$ still contains the ellipse $A(D)$, hence  $I_+ ^2(s\circ A) \leq 1$. Consider the map $T=\det (A) ^{-\frac 1 2} \cdot A \in {\rm SL}_2$. Then $I^2_+(s\circ T) <1$ since $\det (A) >1$. This contradicts the assumption that $s$ is $I^2_+$-minimal.

On the other hand,  if $s$ is isotropic then $\bar D$ is the Loewner ellipse of $B$. Consider an $I^2_+$-minimal norm $s'=s\circ T$ in the ${\rm SL}_2$-orbit of $s$. Then the Loewner ellipse $T(\bar{D})$ of $s'$ must be a multiple of $\bar D$, as we have seen above. Hence $T\in {\rm SO}_2$. Since $I^2_+$ is conformally invariant, we get $I^2_+(s)=I^2_+(s')$, and $s$ is $I^2_+$-minimal.
\end{proof}

Now we can easily deduce:
\begin{cor} \label{corivanov}
For the Reshetnyak definition of energy $\I=I_+^2$ the normalization constant $\lambda _{\I}$ equals $1$, the optimal quasiconformality constant  $Q_{\I}$ equals $\sqrt 2$, and the induced definition of area $\mu ^{\I}$ is the inscribed Riemannian area $\mu ^i$.
\end{cor}

\begin{proof}
We have $\lambda _{\mathcal I} =\frac 1 {\hat J( s_0)}= \frac 1 {\I(s_0)}=1$ since $s_0$ is $I_+^2$-minimal. Isotropic seminorms are
$\sqrt 2$-quasiconformal by  John's   theorem. The supremum norm $s_{\infty} \in \mathfrak S_2$  is isotropic, hence $I^2_+$-minimal. For $s_{\infty}$ the
  quasiconformality constant $\sqrt 2$ is optimal.

In order to prove that the induced definition of area coincides with the  inscribed Riemannian area $\mu ^i$, it suffices to evaluate the Jacobians on any $I^2_+$-minimal norm $s$.  By homogeneity we may assume again that the Loewner  ellipse  of the unit ball $B$ of $s$ is the unit disc
$\bar D$. Then $\J ^{\mathcal I} (s)= 1 =\J^i (s)$.
\end{proof}

\subsection{The Korevaar-Schoen energy and the Dirichlet area} \label{subsec+}
Unfortunately, in the classical case of the Korevaar-Schoen energy $\I=I^2$ we do not know much about the induced definition of  area/Jacobian. We 
call this the Dirichlet definition of area/Jacobian  and denote it by $\mu ^D$ and $\J ^D$, respectively. Only the normalization constant in this case is easy to determine.

\begin{lem}
For the Korevaar-Schoen energy $\I=I^2$, the canonical Euclidean norm $s_0$ is $I^2$-minimal. The normalization
constant $\lambda _{\mathcal I}$ equals $\frac 1 2$.
\end{lem}

\begin{proof}
We have $I^2(s_0)=\frac 1 {\pi} \cdot 2\pi =2$. Therefore,  it suffices to prove the $I^2$-minimality of $s_0$.  Since $I^2$  and $s_0$ are ${\rm SO}_2$-invariant, it suffices to prove $I^2(s_0\circ T) \geq I^2(s_0)$ for any symmetric matrix
$T \in {\rm SL}_2 $. In this case, one easily computes $I^2 (s\circ T) =\frac 1 2 (\lambda _1 ^2 +\lambda _2 ^2)$, where
$\lambda _{1,2}$ are the eigenvalues of  $T$.  Under the assumption $\lambda _1\cdot \lambda _2=\det (T)=1$ the minimum is achieved for $\lambda _1=\lambda _2=1$.  Hence $s_0$ is $I^2$-minimal.
\end{proof}

From the corresponding property of $\I=I^2$, it is easy to deduce that for   norms $s\neq s'$ the inequality  $s\geq s'$ implies the strict inequality  $\J^D(s) >\J^D (s')$, in contrast to the cases of inscribed Riemannian and Benson definitions of areas $\mu ^i$ and $m^{\ast}$. In \cite{LW} it is shown that for $\I=I^2$ the quasiconformality constant
$Q_{\I}$ in \lref{quasicon} can be chosen to be $2\sqrt2 +\sqrt 6$.  However,  the computation of $Q_{\I}$ in \cite{LW}  and  the above strict monotonicity statement show that this constant is not optimal.

 Computing $\J^D$ on the supremum norm $s_{\infty}$  it is possible to see that
$\mu ^D$ is  different from the Busemann and Holmes-Thompson definitions of area.
We leave the lengthy computation to the interested reader.



\section{Main lemma and main theorems}  \label{sec4}
\subsection{Basic observations}  Let $\I$ be a definition of energy and let
$\mu ^{\I}$  and $\J^{\I}$ be the corresponding definitions of area and Jacobian.
 Let $\lambda _{\I}$ be the normalization   constant from the previous section.

Let $X$ be a metric space, $\Omega \subset \R^2$ a domain  and let $u\in W^{1,2} (\Omega ,X)$ be a Sobolev map.     Integrating the point-wise inequality \eqref{generalineq} we deduce:
\begin{equation} \label{genineq}
\Area_{\mu ^{\I}} (u)  \leq \lambda _{\I} \cdot E_{\I} (u)
\end{equation}
Moreover, equality holds if and only if  the approximate metric derivative $ \apmd u_z$  is $\I$-minimal for almost all $z\in \Omega$.
   In case of equality, \lref{quasicon} implies that  the map  $u$ is $Q_{\I}$-quasiconformal.

\subsection{Main Lemma} Conformal invariance  of $\I$ together with the usual transformation rule (\cite{LW}, Lemma 4.9)  has the following direct consequence: For any conformal diffeomorphism $\phi:\Omega '\to \Omega$
which is biLipschitz
and  for any map $u\in W^{1,2} (\Omega ,X)$, the composition  $u\circ \phi$ is contained in $ W^{1,2} (\Omega ', X)$, and it  has the same $\I$-energy as $u$.

The general transformation formula shows that for any definition
of area $\mu$,  any biLipschitz homeomorphism $\phi:\Omega '\to \Omega$,
 and any  $u\in W^{1,2} (\Omega ,X)$ the map $u\circ \phi \in W^{1,2} (\Omega ', X)$ has the same $\mu$-area as $u$.

Now we can state the main technical lemma, which appears implicitly   in
\cite{LW}:

\begin{lem}  \label{mainlem}
Let $\I, \mu ^{\I} , \lambda _{\I}$ be as above.   Let $X$ be a metric space and let $u\in W^{1,2} (D,X)$ be arbitrary.  Then the following conditions are equivalent:
\begin{enumerate}
\item $\Area _{\mu ^{\I} }(u) =\lambda _{\I} \cdot E_{\I}(u)$.
\item For almost every  $z\in D$ the approximate metric derivative $\apmd u_z$ is an  $\I$-minimal seminorm.
\item For every biLipschitz homeomorphism   $\psi : D\to  D$ we have  $E_{\I}(u\circ \psi ) \geq E_{\I}(u)$.
\end{enumerate}
\end{lem}

\begin{proof}
We have already proven the equivalence of (1) and (2).   If (1) holds, then (3) follows directly from the general inequality \eqref{genineq} and invariance of the $\Area _{\mu}$ under diffeomorphisms.

It remains  to prove the main part, namely that  (3) implies  (2).
Thus assume (3) holds.  The conformal invariance of $\I$ and the Riemann mapping theorem
imply that for   any other domain $\Omega \subset \R^2$  with  smooth boundary  and any biLipschitz homeomorphism   $\psi: \Omega \to  D$   the inequality
$E_{\I}(u\circ \psi )\geq E_{\I}(u)$ holds true.  Indeed, we only need to compose $\psi$ with a conformal diffeomorphism $F:D \to \Omega $, which is
biLipschitz since the boundary of $\Omega$ is smooth.

  Assume now that (2) does not hold. Then it is possible  to construct    a biLipschitz map  $\psi$  from a domain $\Omega $ to $D$ such that $E_{\I}(u\circ \psi) <E_{\I}(u)$ in the same way as in the proof of Theorem 6.2  in \cite{LW}, to which we refer for some technical details. Here we just explain the major steps.  First, we find  a compact set $K \subset D$ of positive measure such that at no point $z\in K$ the
approximate metric derivative  $\apmd u_z$ is $\I$-minimal. Making $K$ smaller we may assume  that 
the map $z\mapsto \apmd u_z$ is  continuous on $K$.   By continuity, we find a Lebesgue point $z$ of $K$, a map $T\in {\rm SL}_2 $  and some $\epsilon >0$
 such that  $\I(s\circ T) \leq \I(s)-\epsilon$ for any seminorm $s$ which arises as the approximate metric derivative $\apmd u_y$ at some  point $y\in K\cap B_{\epsilon} (z)$.

  We may assume without loss of generality that $z$ is the origin $0$ and that $T$ is a diagonal matrix with  two different  eigenvalues $\lambda _1 >\lambda _2 =\frac 1  {\lambda _1} >0$.  Then (here comes the trick!) we define a family of biLipschitz homeomorphisms $\psi _r :\R^2\to \R^2$ as follows. The map $\psi _r$ coincides with $T$ on the closed  $r$-ball around $0$.   On the complement of this $r$-ball,
the map $\psi _r$ is the restriction of the holomorphic (hence conformal) map
$f_r:\C^{\ast} \to \C$, defined by $f_r(z)= c \cdot z + r^2 \cdot d \cdot z ^{-1}$, where the constants $c,d \in \C$ are given by
$c=\frac 1 2 (\lambda _1+ \lambda _2)$ and $d =\frac 1 2 (\lambda _1 - \lambda _2)$.  Then  the map $f_r$ coincides with $T$ on the $r$-circle around $0$.
This  map $\psi _r$ is  biLipschitz on $\R^2$ (and smooth outside of the $r$-circle around $0$).  Moreover, the map $\psi _r$  preserves  the $\I$-energy of the map $u$ on the complement of the $r$-ball, due to  the conformality of $f_r$ and the conformal invariance of $\I$.  Finally, by construction of  $T$, the map
$\psi _r$ decreases the $\I$-energy  of $u$ by some positive amount (at least $\frac 1 2 \epsilon \pi r^2$), if $r$ is small enough.

  Thus $E_{\I}(u\circ \psi _r) <E_{\I}(u)$ for $r$ small enough. This provides a contradiction and finishes the proof of the lemma.
\end{proof}

\subsection{Formulation of the main theorems}
The proof of the following  theorem is postponed to the next section.

\begin{thm} \label{quasiarea}
Let $\I$ be a quasi-convex definition of energy. Then the corresponding definition of
 area  $\mu ^{\I}$ is quasi-convex as well.
\end{thm}

\tref{quasiarea}  generalizes the first statement of \tref{thm1}.  Together with
\cref{corivanov} it shows that $\mu ^i$ is quasi-convex, cf. Remark \ref{remark}.

Before turning to the main theorem stating the connection of energy and area minimizers,
we recall an important step in the solution of the Plateau problem (\cite{LW}, Propostion 7.5, \cite{KS93}):
Let  $\Gamma$ be a Jordan curve in a proper metric space  $X$. Assume that the sequence of maps
$w_i \in \Lambda (\Gamma , X)$ has  uniformly bounded Reshetnyak energy $E^2 _+(w_i)$. Then there exist
conformal diffeomorphisms $\phi_i:D\to D$ such that  the sequence $w_i' =w_i\circ \phi  \in \Lambda (\Gamma ,X)$ converges in $L^2$ to a map $\bar w\in \Lambda (\Gamma, X)$.  Note that for any quasi-convex
definition of area $\mu$ or energy $\I$, we have in this case (\cite{LW}, Theorem 5.4):
\begin{equation} \label{limit}
\Area _{\mu} (\bar w) \leq \liminf \Area _{\mu} (w_i)\text{ and }  E_{\I}
(\bar w) \leq \liminf E_{\I}  (w_i).
\end{equation}

The proof of the following  theorem will rely on \tref{quasiarea}.

\begin{thm} \label{thmmain}
Let $\I$ be a  quasi-convex definition of energy. Let $\Gamma$ be a Jordan curve in
a proper metric space $X$.  Any map $u\in \Lambda (\Gamma, X)$ with minimal $\I$-energy in
$\Lambda (\Gamma ,X)$  has   minimal $\mu^{\I}$-area  in $\Lambda (\Gamma ,X)$. Moreover, $u$ is
$Q_{\I}$-quasiconformal.
\end{thm}

\begin{proof}
  Let   $u\in \Lambda (\Gamma, X)$ with minimal $\I$-energy  among all maps $v\in  \Lambda (\Gamma ,X)$ be given. Then $E_{\I}(u) \leq E_{\I}(u\circ \phi)$ for any biLipschitz
	homeomorphism $\phi:D\to D$. Due to \lref{mainlem}, $\Area _{\mu ^{\I} }(u) =\lambda _{\I} \cdot E_{\I}(u)$. Moreover, by \lref{quasicon}, almost all approximate derivatives of $u$
are $Q_{\I}$-quasiconformal. This proves the last statement.



Assume that $u$ does not minimize the $\mu^{\I}$-area and take another element
  $v\in  \Lambda (\Gamma ,X)$ with $\Area _{\mu ^{\I}} (v) <\Area _{\mu ^{\I}}(u)$.   Consider the
set  $\Lambda _0$ of elements $w\in \Lambda (\Gamma ,X)$ with $\Area _{\mu ^\I} (w) \leq \Area _{\mu ^\I} (v)$.
  We take   a sequence $w_n \in \Lambda _0$ such that
$E_{\I} (w_n)$ converges to the infimum  of the $\I$-energy on  $\Lambda _0$.
 Due to \eqref{2ener}, the Reshetnyak energy of all maps $w_n$  is bounded  from above by a uniform constant.
Using the observation preceding \tref{thmmain}, we  find   some   $\bar w \in \Lambda (\Gamma ,X)$ which satisfies \eqref{limit}.  Here we have used the quasi-convexity of $\mu^{\I}$, given   by \tref{quasiarea}.

Thus,  $\bar w$ is contained in $\Lambda _0$ and minimizes the $\mathcal I$-energy in $\Lambda _0$.
In particular,  $E_{\I}(\bar w \circ \phi) \geq E_{\I}(\bar w)$, for any biLipschitz homeomorphism $\phi : D\to  D$.
Applying \lref{mainlem}  to the map $\bar w$  we deduce
$$\lambda _{\I} \cdot E_{\I}(u)= \Area _{\mu^{\I}} (u) >\Area _{\mu ^{\I}}(v)\geq  \Area _{\mu ^{\I}}(\bar w) =\lambda _{\I} \cdot E_{\I}(\bar w).$$
This contradicts the minimality of $E_{\I}(u)$.
\end{proof}

\subsection{Regularity of energy minimizers}
The regularity of energy minimizers is now a direct consequence of \cite{LW}. Recall that a Jordan curve  $\Gamma\subset X$ is a  \emph{chord-arc curve} if the restriction of the metric to $\Gamma$  is biLipschitz equivalent to the induced  intrinsic metric. A map $u:D\to X$ is said to satisfy \emph{Lusin's property $(N)$} if for any subset $S$ of $D$ with area $0$ the image
$u(S)$ has zero two-dimensional Hausdorff measure.

\begin{thm} \label{thmgenreg}
Let $X$ be a  proper metric space which satisfies a uniformly  local quadratic isoperimetric inequality.
Let $\I$ be a quasi-convex definition of energy and let  $\Gamma $ be a  Jordan curve in $X$ such that the set
$\Lambda (\Gamma ,X)$ is not empty. Then there exists a minimizer $u$ of the $\I$-energy in $\Lambda (\Gamma ,X)$.
Any such minimizer has a unique
     locally Hoelder continuous representative which extends to a continuous map on $\bar D$.  Moreover, $u$ is contained in the Sobolev space
     $W^{1,p} _{loc} (D,X)$ for some $p>2$ and satisfies Lusin's property $(N)$.  If the curve $\Gamma$ is a chord-arc curve
     then $u$ is Hoelder continuous on $\bar D$.
\end{thm}

\begin{proof}
The existence of a minimizer $u$ of the $\I$-energy in $\Lambda (\Gamma ,X)$ is a consequence of \cite{LW}, Theorem~5.4 and Proposition~7.5,  see also Theorem~7.6.

Any map $u$ minimizing the $\I$-energy in $\Lambda (\Gamma, X)$  is quasiconformal  and minimizes the $\mu ^{\I}$-area in
$\Lambda (\Gamma, X)$, by \tref{thmmain}. The result now follows from \cite{LW}, Theorems 8.1, 9.2, and 9.3.
\end{proof}

\subsection{Optimal regularity}
We are going to provide the proof of \tref{optregul}. Thus, let $u\in W^{1,2}(D, X)$ be as in \tref{optregul}. Then for any biLipschitz homeomorphism $\psi: D\to D$ we have $\Area_\mu(u\circ \psi)= \Area_\mu(u)$ and therefore $E_+^2(u\circ\psi) \geq E_+^2(u)$. Applying \lref{mainlem} and \lref{resiso} we see that $u$ is infinitesimally isotropic in the following sense.

\begin{defn}
 A map $u\in W^{1,2}(D, X)$ is infinitesimally isotropic if for almost every $z\in D$ the approximate metric derivative of $u$ at $z$ is an isotropic seminorm.
\end{defn}

\tref{optregul} is thus an immediate consequence of the following theorem.

\begin{thm}
 Let $\Gamma$ be a Jordan curve in a metric space $X$. Assume that $X$ satisfies the $(C,l_0,\mu )$-quadratic  isoperimetric inequality
  and let $u\in\Lambda (\Gamma ,X)$ be an infinitesimally isotropic map having minimal $\mu$-area in $\Lambda (\Gamma ,X)$.  Then  $u$ has a
  locally $\alpha$-Hoelder continuous representative  with  $\alpha = q(\mu)\cdot\frac{1}{4 \pi C}$.
\end{thm}

\begin{proof} 
Due to \cite{LW},  $u$ has a unique continuous representative.
For any subdomain $\Omega $ of $D$ we have
 \begin{equation} \label{verylast}
 E_+^2 (u|_{\Omega }) =\Area _{\mu ^i} (u|_{\Omega} )
 \end{equation}
 by \lref{mainlem}.
 Looking into the proof of the Hoelder continuity of $u$ in \cite{LW}, Proposition 8.7,
 we see that the quasiconformality factor $Q$ of $u$ (which, as we know, is bounded by $\sqrt 2$) comes into the game only once. Namely, this happens  in the
  estimate (40) in Lemma 8.8, where the inequality  $E_+^2 (u|_{\Omega} ) \leq Q^2\cdot \Area _{\mu} (u|_{\Omega }) $ appears for open balls $\Omega\subset D$.

  Using \eqref{verylast} together with \eqref{areafunct} we can replace this estimate (40) by
   $$E_+^2 (u|_{\Omega})=\Area _{\mu ^i}  (u|_{\Omega} ) \leq q(\mu)^{-1} \cdot \Area _{\mu}  (u|_{\Omega}). $$

   Hence we can replace the factor $Q^2$ in the proof  of \cite{LW}, Proposition 8.7 by the factor $q(\mu )^{-1}$. Leaving the rest of that  proof  unchanged,  we get  $\alpha = q(\mu)\cdot \frac {1} {4\pi C}$ as a bound for the Hoelder exponent of $u$.
\end{proof}

\section{Quasi-convexity  of  $\mu ^{\I}$}  \label{sec5}
This section is devoted to the
\begin{proof}[Proof of \tref{quasiarea}]
Assume on the contrary, that the definition of energy $\I$ is  quasi-convex, but that $\mu ^{\I}$ is not quasi-convex.  Consider  a finite-dimensional normed vector space $Y$, a linear map $L: \R^2 \to Y$ and a sequence of
smooth embeddings $\psi _m: \bar D\to Y$ as in Subsection \ref{subsecquas}, such that the following holds true.  The maps $\psi _m$ coincide with $L$ on the boundary circle $S^1$, they  are $K$-biLipschitz with a fixed constant $K$, and  they converge uniformly to the restriction of $L$ to $\bar D$.   Finally, for some $\epsilon >0$ and all $m>0$, we have
$$\Area_{\mu^{\I}} (L|_D)  \geq  \Area _{\mu ^{\I}} (\psi _m)  +\epsilon.$$
We will use this sequence to obtain a contradiction to the semi-continuity of $E_{\I}$. The idea is to modify $\psi _m$ by (almost) homeomorphisms, so that the new maps satisfy equality in the  main area-energy inequality \eqref{genineq}. We explain this modification   in a slightly more abstract context of general biLipschitz discs.

 The first observation is a direct consequence of the fact that the diameter of a simple closed curve in
 $\R^2$ equals  the diameter of the corresponding Jordan domain.

\begin{lem} \label{pseudo}
Let $Z$ be a  metric space which is $K$-biLipschitz to the disc $\bar D$ and let
 $u:\bar D\to Z$ be any homeomorphism. Then for any Jordan curve $\gamma \subset \bar D$ and the corresponding Jordan domain $J\subset \bar D$
we have $\diam (u(\gamma) ) \geq K^2 \cdot \diam (u(J))$.
%
%
%
\end{lem}

By continuity, the same inequality holds true for any uniform limit of homeomorphisms from $\bar D$  to $Z$, the class of maps we are going to
consider now more closely.    Let again the space $Z$ be $K$-biLipschitz to
the unit disc, let  us fix three distinct points $p_1,p_2,p_3$ on $S^1$ and three distinct points $x_1,x_2,x_3$ on the boundary circle  $\Gamma$ of
$Z$.  Let $\Lambda _0 (Z)$ denote the set of all continuous maps $u:\bar D\to Z$,
which send $p_i$ to $x_i$,  which are uniform limits of homeomorphisms from $\bar D$ to  $Z$, and whose restrictions to $D$ are contained in the Sobolev space $W^{1,2} (D,Z)$.

As uniform limits of homeomorphisms, any  map  $u\in \Lambda _0 (Z)$ has the whole set $Z$ as its image.
When applied to all circles $\gamma $ contained in $D$, the conclusion of \lref{pseudo}  shows that any  $u\in \Lambda _0(Z)$
is $K^2$-pseudomonotone  in the sense of  \cite{MM}. Fixing a biLipschitz homeomorphism
$\psi:\bar D \to Z$, we see that $\psi ^{-1} \circ u:\bar D\to \bar D$ is pseudomonotone as well.
 Using   \cite{MM}, we deduce that $\psi ^{-1} \circ u$ satisfies Lusin's property  (N),
for  any  $u\in \Lambda _0 (Z)$. Hence, any $u\in \Lambda _0 (Z)$   satisfies Lusin's property (N)  as well. See also  \cite{Kar07}, Theorem 2.4.


%

\begin{lem} \label{samemes}
For all elements $u\in\Lambda _0 (Z)$ the value  $\Area_{\mu ^{\I}} (u)$ is independent of the choice of $u$.
\end{lem}

\begin{proof}
Fix again the biLipschitz homeomorphism $\psi:\bar D\to Z$ and consider $v=\psi ^{-1} \circ u \in \Lambda _0 (\bar D)$.
Since $v$  is a uniform limit of homeomorphisms, any fiber of $v$ is a cell-like set (\cite{HNV04}, p.97), in particular, any such fiber is connected. Applying the area formula to the continuous Sobolev map
$v: D\to \bar D$ which satisfies Lusin's property (N) (cf. \cite{Kar07}), we see that for almost all
$z\in D$ the preimage $v^{-1} (z)$ has only finitely many points. By the connectedness of the fibers, we see that almost every fiber $v^{-1} (z)$ has exactly  one point.
Now we see:
$$\Area_{\mu ^{\I}} (u) =  \int_D \J^{\I} (\apmd u_z)dz =\int_D |\det(d_zv)| \J^{\I} (\md\psi_{v(z)})dz.$$
The area formula for  the Sobolev map $v:D\to D$ (\cite{Kar07}) gives us:

$$\Area _{\mu^{\I}} (u)=
\J^{\I} (\md\psi_y)dy  = \Area _{\mu^{\I}}(\psi).$$
%
\end{proof}

The next lemma is essentially taken from \cite{Jost}:
\begin{lem} \label{equicon}
For any $C>0$, the set $\Lambda _0 ^C (Z)$ of all elements $u$ in $\Lambda_0 (Z)$
with $E^2 _+ (u) \leq C$ is equi-continuous.
\end{lem}

\begin{proof}
The equi-continuity of the restrictions of $u$ to the boundary circle $S^1$ is
 part  of the classical solution of the Plateau problem, see \cite{LW}, Propostion 7.4.
By the Courant-Lebesgue lemma (\cite{LW}, Lemma  7.3), for any $\epsilon>0$ there is some $\delta =\delta (\epsilon, C)$    such that for
 any $x\in \bar D$  and any $u\in \Lambda ^{C} _0 (Z)$ there is some
$\sqrt {\delta }> r>\delta$ such that $\partial B_r (x) \cap \bar D$ is mapped by $u$ to a curve of diameter $\leq \epsilon$.

If $B_{\delta} (x)$  does not intersect the boundary circle $S^1$ then $u(B_{\delta} (x))$ has diameter $\leq  K^2\cdot  \epsilon$ by \lref{pseudo}.  On the other hand, if  $B_{\delta}  (x)$ intersects
$S^1$, then we see that the image of the intersection of $B_{\delta} (x)$ with $S^1$ has diameter bounded as well  by some $\epsilon ' >0$  depending only on $\delta $ and going to $0$ with $\delta $, due to the equi-continuity of the restrictions $u|_{S^1}$. We may assume $\epsilon=\epsilon '$.
Then the Jordan curve consisting of the corresponding parts of $\partial B_{\delta}  (x)$ and boundary $S^1$  has as its image a curve of diameter at most $2 \epsilon$. Thus
using the biLipschitz property of $Z$ as in \lref{pseudo}, we see that
the ball $B_{\delta} (x)$ is mapped onto a set of diameter $\leq 2K^2 \cdot \epsilon$.
\end{proof}

 The proof above shows that the modulus of continuity of any $u\in \Lambda _0 ^C (Z)$  depends only on the constants $C,K$, the boundary
circle $\Gamma \subset Z$ and the choice of the fixed points $x_i \in \Gamma$.

\begin{cor} \label{corfin}
There is a map $u\in \Lambda _0 (Z)$ with minimal $\I$-energy in $\Lambda _0 (Z)$. This element $u$ satisfies $\Area  _{\mu ^{\I}} (u) =\lambda _{\I} \cdot E_{\I}(u)$.
\end{cor}

\begin{proof}
Take a sequence $u_n \in \Lambda _0  (Z)$  whose $\I$-energies converge to
the infimum of $\I$ on $ \Lambda _0  (Z)$. By \eqref{2ener}, $E^2_+$ is bounded by a multiple of $\I$. Therefore, we can apply \lref{equicon} and deduce that the sequence $u_n$ is equi-continuous. By Arzela-Ascoli,  we find a  map $u:\bar D\to Z$ as a uniform limit of a subsequence of the $u_n$. This map $u$ is a uniform limit of uniform limits
of homeomorphisms, hence $u$ itself is a uniform limit of homeomorphisms.
Moreover, $u (p_i)=x_i$ for $i=1,2,3$.  Finally, the map is contained in $W^{1,2} (D,X)$
as an $L^2$-limit of Sobolev maps with uniformly bounded   energy, hence
$u\in \Lambda _0 (Z)$. Since    $\I$ is quasi-convex,  we have $E_{\I} (u) \leq \lim _{n\to \infty} E_{\I}(u_n)$, see \cite{LW}, Theorem 5.4. Therefore, $u$ has minimal $\mathcal I$-energy in $\Lambda _0  (Z)$.

 If $\phi:\bar D\to \bar D$ were a biLipschitz homeomorphism with
$E_{\I}(u\circ \phi )< E_{\I}(u)$  we would consider a M\"obius map $\phi _0:\bar D\to \bar D$,
such that $\phi \circ \phi _0$ fixes the points $p_i$. Then the map $u':= u\circ \phi \circ \phi _0$ is in $\Lambda _0 (Z)$ and has the same $\I$-energy as
$u\circ \phi$, due to the conformal invariance of $\I$. This would contradict the minimality of
$E_{\I}(u)$ in $\Lambda _0 (Z)$.     Hence  such a homeomorphism $\phi$ cannot exist and we may apply \lref{mainlem}, to obtain the equality $\Area _{\mu^{\I}} (u) =   \lambda _{\I} \cdot E_{\I}(u)$.
\end{proof}

Now it is  easy to use $\psi _n$ to obtain a contradiction to the quasi-convexity
of $\I$.    Denote by $Z_n$ the image $\psi _n (\bar D)$ and by $Z$ the ellipse
$L(\bar D)$.  By construction, all $Z_n$ and $Z$  are $K$-biLipschitz to $\bar D$ and share the same boundary circle.  We denote it by $\Gamma$ and fix the same
triple $x_1,x_2,x_3$ in $\Gamma$ for all $Z_n$ and $Z$.

  Consider a map $v_n \in \Lambda _0 (Z_n)$ with minimal $\I$-energy in
$\Lambda_0 (Z_n)$.  By \cref{corfin}, such $v_n$ exists and satisfies
$\Area_{\mu ^{\I}}(v_n) =\lambda  _{\I}  \cdot E_{\I}(v_n)$.  Moreover, by \lref{equicon} and the subsequent observation, the maps $v_n$ are equi-continuous. Finally,  by \lref{samemes}, we have
$\Area _{\mu ^{\I}} (v_n)= \Area _{\mu ^{\I}} (\psi _n)$.

The images of the maps $v_n:\bar D\to Z_n\to Y$ are contained in a compact set.
Hence, by Arzela-Ascoli after choosing a subsequence, the maps $v_n$ uniformly converge to a map $v:\bar D\to Z$. Moreover, identifying $Z_n$ with $Z$ by some uniformly biLipschitz homeomorphisms  point-wise converging to  the identity of $Z$, we see that the limiting map $v$ can be represented  as a uniform limit of homeomorphisms from $\bar D$ to  $Z$.    Since the $v_n$ have uniformly bounded energies, the limit map $v$
lies in  the Sobolev class $W^{1,2} (D,Z)$.    Thus, by construction, $v \in \Lambda _0 (Z)$.
Finally, by the semi-continuity of $\I$, we must have
$E_{\I}(v) \leq
\liminf _{n\to \infty} E_{\I}(v_n)$.

Taking all inequalities together we get for large $n$:
\begin{equation*}
 \begin{split}
  \Area _{\mu ^{\I}} (v)&=\Area_ {\mu ^{\I}} (L|_D )\geq \Area _{\mu ^{\I}}(\psi _n) +\epsilon  = \lambda _{\I} \cdot E_{\I}(v_n) +\epsilon\\
   &  \geq \lambda _{\I} \cdot E_{\I}(v) + \frac 1 2 \epsilon.
  \end{split}
  \end{equation*}
But this contradicts the main inequality \eqref{genineq}  and finishes the proof of \tref{quasiarea}.
\end{proof}

\providecommand{\bysame}{\leavevmode\hbox to3em{\hrulefill}\thinspace}
\providecommand{\MR}{\relax\ifhmode\unskip\space\fi MR }
\providecommand{\MRhref}[2]{%
  \href{http://www.ams.org/mathscinet-getitem?mr=#1}{#2}
}
\providecommand{\href}[2]{#2}

\end{document}